\documentclass[12pt,a4paper]{article}
\usepackage{authblk}

\usepackage[latin1]{inputenc}
\usepackage[english]{babel}
\usepackage{amsfonts,amsmath}
\usepackage{mathabx}
\usepackage{fancyhdr,graphicx, xcolor}

\usepackage{caption}
\usepackage{ulem} 

\usepackage{mathtools}

\usepackage[shortlabels]{enumitem}

\usepackage{amsthm} 

\usepackage[scanall,209mode]{psfrag}

\def\calB{\mathcal B}
\def\calD{\mathcal D}
\def\DS{\mathcal{D}_S}

\def\R{\mathbb R}
\def\C{\mathbb C}

\def\ran{ \mathop{ran} }

\DeclareMathOperator{\Diag}{Diag}

\def\tr{\text{tr}}
\def\im{\text{im}}
\def\ran{\text{ran}}

\def\restrict#1{\raise-.5ex\hbox{\ensuremath{\big|}}_{#1}} 

\usepackage[hidelinks]{hyperref} 
\newtheorem{theorem}{Theorem}

\newtheorem{corollary}{Corollary}

\newtheorem{definition}{Definition}
\newtheorem{example}{Example}

\newtheorem{lemma}{Lemma}

\newtheorem{proposition}{Proposition}
\newtheorem{remark}{Remark}

\numberwithin{equation}{section}

\title{Best approximants relative to a C$^*$-subalgebra, joint numerical range and subdifferentials}

  \author[1,2]{Tamara Bottazzi}
  \author[3,4]{Alejandro Varela}
  
  \affil[1]{Universidad Nacional de R\'io Negro. Centro Interdisciplinario de Telecomunicaciones, Electrónica, Computación y Ciencia Aplicada, Sede Andina (8400) S.C. de Bariloche, Argentina}
  \affil[2]{Consejo Nacional de Investigaciones Cient\'ificas y T\'ecnicas, (1425) Buenos Aires,
  	Argentina}
  \affil[3]{Instituto de Ciencias, Universidad Nacional de Gral. Sarmiento, J.	
  	M. Gutierrez 1150, (B1613GSX) Los Polvorines, Argentina}
  
  \affil[4]{Instituto Argentino de Matem\'atica ``Alberto P. Calder\'on", Saavedra 15 3er. piso, (C1083ACA) Buenos Aires, Argentina}
  
  \affil[ ]{\texttt{tbottazzi@unrn.edu.ar} \quad \texttt{avarela@campus.ungs.edu.ar}}
  
\begin{document} 
\date{}
\maketitle

%
%
%
%

\begin{abstract}
	We study the minimality of $n\times n$ Hermitian matrices $A$ respect to a $C^*$-subalgebra $\mathcal{B}$ of $M_n(\C)$ in the spectral norm, that is
	\[\|A\|\leq \|A+B\|,\ \text{ for every } B\in \mathcal{B}.\]
	We generalize the notion of the moment of a subspace and relate it to the joint numerical range and the subdifferentials of the maximum eigenvalue. We extend results previously known for the subalgebra of diagonal operators and describe the subdifferential of the maximum eigenvalue in terms of the moment of the corresponding eigenspace. We also characterize $\mathcal{B}$-minimality via moments and subdifferentials, and provide examples.
\end{abstract}

{\bf 2020 MSC: (PRIMARY) 15A60, 41A50, 47B15, (Secondary) 47A12, 47A30}


\section{Introduction and prelimininaries}

Let $M_n(\C)$ be the space of complex square matrices of $n\times n$, which is a $C^*$-algebra. Consider a $C^*$-subalgebra $\calB$ of $M_n(\C)$. We call $A\in M_n(\C)$ a $\calB$-minimal matrix if
\[\|A\|\leq \|A+B\|, \text{ for all } B\in \calB.\]
$\calB$-minimality is closely related with best approximation, since $A$ is $\calB$-minimal if and only if 
\[dist(A,\calB)=\|A\|.\]
If $A$ is Hermitian, then 
\[dist(A,\mathcal{B})=dist(A,\mathcal{B}^h),\]
where the superscript denotes the subset of Hermitian (self-adjoint) elements. This holds by Remark 2.1 \cite{ammrv}, $\|A+\text{Re} (B)\|=\|\text{Re}(A+B)\|\leq \|A+B\|$ for all $B\in \mathcal{B}$. So, we will consider only the Hermitian elements of $\mathcal{B}$.

Our interest in this subject is related to the following context. Let $\mathcal{A}$ be a C$^*$-algebra and $\mathcal{B}$ a C$^*$-subalgebra sharing the same unit. The description of minimal self-adjoint elements $B\in\calB$ with respect to a self-adjoint element $A\in \mathcal{A}$ ($\calB$-minimality) allows a concrete description of geodesics in a natural Finsler metric induced by the norm of $\mathcal{A}$ on the flag manifolds $U_\mathcal{A}/U_\mathcal{B}$ (a homogeneous space under the left action of $U_\mathcal{A}$) associated with the quotient space $\mathcal{A}/\mathcal{B}$ (see \cite{dmr1}). 
In this setting, if $Z$ is a Hermitian matrix that is minimal in the spectral norm, then the curve $\gamma(t)=e^{itZ}A e^{-itZ}$ is length-minimizing in $U_\mathcal{A}/U_\mathcal{B}$ between $\gamma(0)$ and $\gamma(t)$ for $|t|\leq \frac{\pi}{2\|Z\|}$ (see \cite{dmr1}).

$\calB$-Minimality is closely related to Birkhoff-James orthogonality in the normed space $\left( M_n(\C), \|\cdot\|\right)$ \cite{Bi,James}: two matrices $A,B\in M_n(\C)$ are Birkhoff-James orthogonal if and only if $\|A\|\leq \|A+\lambda B\|$, for every $\lambda\in \C$. In particular, $A\in M_n(\C)$ is a $\calB$-minimal matrix if and only if $A$ is Birkhoff-James orthogonal to every $B\in \calB$. This is equivalent to state that $A$ is Birkhoff-James orthogonal to $\calB$ \cite{Grover, B}.

The case where $\calB$ is the $C^*$-subalgebra of diagonal matrices (respect to a fixed basis of $\C^n$) has been extensively studied in \cite{alrv, BV-2013, KV 3x3, BV-Moment, BV-SDP}. In the present work, we are focused to study and generalize some of the results that involve relations between $\calB$-minimality, joint numerical range (usually called algebraic joint numerical range) and subdifferentials, as we did for diagonal matrices in \cite{KV moment, BV-Moment, BV-SDP}. In this sense, in \cite{BV-SDP} we relate minimal operators to the subdifferentials of the maximum eigenvalue and the norm. We connect these notions with the moments of the eigenspaces corresponding to the largest and smallest eigenvalues, as well as with certain joint numerical ranges; these ideas were developed in \cite{KV moment} for matrices, and \cite{BV-Moment} for compact operators on a Hilbert space. In particular, one of our main results provides an explicit characterization of the subdifferential of the largest eigenvalue of a Hermitian matrix $A(x)$ with variable real diagonal $x$,  
\begin{equation*}
	\partial\big(\lambda_{max}(A(x))\big)=\Diag(\partial\lambda_{max}(A(x)))= m_{S_{max}},
\end{equation*}
in terms of $m_{S_{max}}=\text{conv}\{(|v_1|^2,..., |v_n|^2): v= (v_1,...,v_n)\in S_{\lambda_{max}}, \|v\|=1\}$, the moment of the eigenspace related to the largest eigenvalue $\lambda_{max}(A(x))$. 

Now, considering any other $C^*$-subalgebra $\calB$ of $M_n(\C)$, we need to generalize the definition of moment of a subspace in order to study its relation with the joint numerical range of a certain family and the subdifferentials of the maximum eigenvalue of a variable matrix $A(x)=A_0+\sum_{i=1}^{\dim(\calB)}x_iB_i$ with $A_0\in M_n^h(\C)$, $x_i\in \R$ and $\{B_i\}_{i=1}^{\dim(\calB)}$ any fixed basis of $\calB$.  

The results presented in this paper are divided into three parts. 
In Section \ref{sec: grales}, we establish the appropriate definitions for the generalization we wish to develop and obtain several results that were previously known only for the algebra of diagonal operators. These include the moment set of an eigenspace, the joint numerical range of certain matrices defined with respect to a chosen basis $\mathbb{B}$ of $\calB$, the support of a pair of subspaces and their relation with $\calB$-minimality.
In Section \ref{section subdiff and moment} we study the subdifferential of the maximum eigenvalue of matrices and we link it with the moment of its eigenspace. Moreover, we state an equivalence between $\calB$-minimality and the condition $0\in \partial\lambda_{min}+\partial\lambda_{max}$.
Finally, Section \ref{sec examples} is dedicated to show different examples, including one related to quantum information theory.

\section{General framework of $\calB$-minimality}\label{sec: grales}

Let $\calB$ be a C$^*$-subalgebra of $M_n(\C)$, and the subspace
$$
\mathcal{B}_{\tr}^\perp=\{X\in M_n(\C):\tr(XB)=0, \forall B\in\calB\}
$$
with $\dim(\calB_{\tr})^\perp=n^2-\dim(\calB)$.

Recall the following result from \cite[Proposition 2.3, Lemma 2.5 and Theorem 2.6]{zhang-jiang-han-matrices}.

\begin{theorem}\label{teo X mod}
Let $\mathcal{B}$ be a C$^*$-subalgebra of $M_n(\C)$.  Then
\begin{equation}
	\begin{split}
		A\in M_n^h(\C)& \text{ is } \mathcal{B}\text{-minimal } \Leftrightarrow \\
		&\Leftrightarrow 
		\text{ there exists a Hermitian } X \in \mathcal{B}_{\tr}^\perp\setminus\{0\} \\
		&\hskip.7cm\text{ such that } A X =\|A\|\ |X|.
	\end{split}
\end{equation}

Moreover, if $A \in M_n^h(\C)$ is $\mathcal{B}$-minimal, and $\mathcal{B}$
shares the identity $I=I_n$ with $M_n^h(\C)$,
 then $\pm \|A\| \in \sigma(A)$.
\end{theorem}
With the hypothesis and notations of Theorem \ref{teo X mod} plus the condition that $\mathcal{B}$ has the same identity that $M_n(\C)$, we can consider the spectral projections $E_+$, $E_-$, $I-E_+-E_-$ and the orthogonal decomposition of $\C^n=\text{ran}(E_+)\oplus \text{ran}(E_-)\oplus \big(\text{ran}(E_+)\oplus \text{ran}(E_-)\big)^\perp$ in terms of the eigenspaces of $A$ corresponding to the eigenvalues of $-\|A\|$ and $\|A\|$. Hence, supposing that $\|A\|=1$, we can write  
\begin{equation}
	\label{equ A y X en descomp ortog de A}
A=\begin{pmatrix}
	I&0&0\\
	0&-I&0\\
	0&0&A_{3,3} 
\end{pmatrix}
\ \text{ and }\
X=\begin{pmatrix}
	X_{1,1}&0&0\\
	0&X_{2,2}&0\\
	0&0&0
\end{pmatrix}.
\end{equation}
with $X_{1,1}> 0$ and $X_{2,2}< 0$ (see  Case 1 in Theorem 2.7 of \cite{zhang-jiang-han-matrices} and note that $X_{1,1}$ and $X_{2,2}$ can be diagonal if certain basis is chosen). 
Here $X \ge 0$ denotes a positive semidefinite matrix, that is, $\langle Xu,u\rangle \ge 0$ for all $u \in \mathbb{C}^n$. We write $X > 0$ (respectively $X < 0$) to indicate that $X$ is positive (respectively negative) definite.

Let us denote as $\calD$ the set of density matrices, that is
	\[\calD=\{X\in M_n^h(\C):\ X\geq 0,\ \tr(X)=1 \}.
	\]

Now suppose that $\{B_1, B_2, \dots, B_r\}$ is a fixed orthonormal basis of $\mathcal{B}$. Then, for any $Y \in M_n(\mathbb{C})$, it is clear that
\begin{equation*}
	\begin{split}
\operatorname{tr}(Y B_k) &= 0 \ \text{for all } k = 1, \dots, r 
\ \Leftrightarrow \\
& \Leftrightarrow\operatorname{tr}(Y B) = 0 \ \text{for all } B \in \mathcal{B} 
\ \Leftrightarrow\ 
Y \in \mathcal{B}^\perp_{\operatorname{tr}}.
	\end{split}
\end{equation*}

\begin{remark}
	If $I\in\calB$ and $\dim\calB=t$ we can consider an orthonormal basis of $\calB$ as $\mathbb{B}=\{\frac1{\sqrt{n}} I,B_2,\dots,B_t \}$ with $B_j$ Hermitian. Observe that then it must hold that $\tr(B_j)=0$ for $j=2,\dots, t$. $\mathbb{B}$ is also a orthonormal basis of $\calB^h$ as a real subspace. See for example the Gell-Mann generalized basis used in \cite[proof of Proposition 9]{BV-Moment}.
\end{remark}

In particular, taking the $X$ of Theorem \ref{teo X mod} (see \eqref{equ A y X en descomp ortog de A}), then $0=\tr(X I)=\tr(X_{1,1}+X_{2,2})$ and hence $\tr(X_{1,1})=-\tr(X_{2,2})\neq 0$ since $X\neq 0$.
Then, if we denote with $\rho_1=\frac1{\tr(X_{1,1})}X_{1,1}$ and $\rho_2=\frac1{\tr(X_{2,2})}X_{2,2}$, follows that $\rho_1,\rho_2\geq0$ and $\tr(\rho_1)=\tr(\rho_2)=1$, and hence $\rho_1,\rho_2\in\calD$. Moreover, since $X=X_{1,1}+X_{2,2}\in \calB_{\tr}^\perp$, $\tr\big((\rho_1-\rho_2)B\big)=\frac1{\tr(X_{1,1})}\tr\big((X_{1,1}+X_{2,2})B\big)=0$ for all $B\in\calB$, we obtain that $\tr(\rho_1 B)=\tr(\rho_2 B)$ for all $B\in\calB$, or equivalently
$$
\tr(\rho_1 B_k)=\tr(\rho_2 B_k) \ , \forall k=1,\dots, r 
$$
for $\{B_k\}_{k=1}^r$ any orthonormal basis of $\calB$.

Note that, since $\rho_1 , \rho_2\geq0 $, then $\tr(\rho_1 B_k)=\tr(\rho_2 B_k)\in\R$.

Recall the definition of the (algebraic) joint numerical range of a family of $t$ Hermitian matrices $A_1, \dots, A_t\in M_n^h(\C)$ which is the convex set of $\R^t$
\begin{equation} \label{def jnr}
	\begin{split}
		W(A_1, \dots, A_t) = \{  (\tr(\rho A_1),   \dots, \tr(\rho A_t))& \in \R^t:	\, \rho \in \calD \}.
	\end{split}
\end{equation}

\begin{definition}\label{defi mS}
 	Let $S$ be a subspace of $\C^n$ with $\dim(S)=r$. We define the moment $m_S$ of $S$ in a fixed orthonormal basis $\mathbb{B}=\{B_k\}_{i=1}^{\dim(\calB)}$ of Hermitian matrices of a C$^*$-subalgebra $\calB\subset M_n(\C)$, as the subset of $\R^{\dim(\calB)}$ given by
 	\begin{equation}
 		\label{def momento de un Subesp en relacion a B}
 		m_S=m_{S,\mathbb{B}}=\mathop{\bigcup}\limits_{\substack{\text{o.n. basis } \\
 				\{v_i\}_{i=1}^{r} \text{of } S}
 		}  \text{conv}\left\{ \left(\left(v_i v_i^*\right)_1,\dots, \left(v_i v_i^*\right)_{\dim(\calB)} \right)  \right\}_{i=1}^{r}
 	\end{equation}
 	where we denoted with $v_i$ the column vector, $v_i^*$ the row vector conjugated, $\left(v_i v_i^*\right)_k=\langle v_i v_i^*, B_k\rangle=\tr(v_i v_i^* B_k^*)=\tr( v_i^* B_k v_i)\in\R$ the $k^\text{th}$ coordinate of the orthogonal projection  $v_i v_i^*$ of  rank one in the basis $\mathbb{B}$ of $\calB$. See equivalent description of $m_S$ in Proposition \ref{prop 1}.

\end{definition}

Let 
\begin{equation}
	\label{def matrices de densidad de S}
	\calD_{S}=\{\rho\in M^h_n(\C): \im(\rho)\subset S,\ \rho\geq 0 \text{ and } \tr(\rho)=1\}\subseteq \calD
\end{equation}
be the set of density matrices of a subspace $S\subset \C^n$.

For every $\rho\in\calD_S$, holds that 
\begin{equation}\label{equ densidad escrita usando autovals y autovecs}
\rho=\sum_{i=1}^{\dim(S)} \lambda_i s_i (s_i)^*
\end{equation}
for $\lambda_i\in\R_{\geq 0}$, with $\sum_{i=1}^{\dim(S)} \lambda_i =1$ and $\{s_i\}_{i=1}^{\dim(S)}$ an o.n. basis of $S$. 

Consider the application $\Phi^\mathbb{B}:\calD_{S}\to\R^{\dim(\calB)}$ defined by
\begin{equation}\label{defi phi}
	\Phi^{\mathbb{B}}(\rho)= (\rho_1, \rho_2, \dots, \rho_{\dim(\calB)} ),
\end{equation}
where $\rho_k=\langle \rho,B_k\rangle=\tr(\rho B_k)$ is the $k^{\text{th}}$ coordinate of the projection of $\rho$ on $\calB$ in the orthonormal basis $\mathbb{B}=\{B_k\}_{k=1}^{\dim(\calB)}$ of $\calB$. 

\begin{remark}\label{change of basis}
	If $\mathbb{B}=\{B_k\}_{k=1}^{\dim(\calB)}$ and $\mathbb{E}=\{E_k\}_{k=1}^{\dim(\calB)}$ are two different orthonormal basis of Hermitian matrices of $\calB$ and $\dim(\calB)=s$, then for any $\rho\in \calD_{S}$
	\[\rho_k= \langle\rho,B_k\rangle= \langle \rho, \sum_{i=1}^{s}\beta_i^k E_i \rangle=\sum_{i=1}^{s}\overline{\beta_i^k}\left\langle \rho,  E_i \right\rangle, \]
	where $B_k=\sum_{i=1}^{s}\beta_i^k E_i$ with $\beta_i^k=\langle B_k,E_i\rangle=\tr(B_kE_i)=\tr(E_iB_k)=\langle E_i,B_k\rangle\in \R$, $1\leq i\leq s$. Therefore, 
	\begin{align*}
		\Phi^{\mathbb{B}}(\rho)&= (\rho_1, \rho_2, \dots, \rho_{s} )\\
		&= \left(\sum_{i=1}^{s}\beta_i^1\left\langle \rho,  E_i \right\rangle ,\sum_{i=1}^{s}\beta_i^2\left\langle \rho,  E_i \right\rangle,...\sum_{i=1}^{s}\beta_i^{s}\left\langle \rho,  E_i \right\rangle\right) \\
		&=\begin{pmatrix}
			\beta_1^1&\beta_2^1&\cdots&\beta_{s}^1\\
			\beta_1^2&\beta_2^2&\cdots&\beta_{s}^2\\
			\vdots&\vdots&\vdots&\vdots\\
			\beta_1^{s}&\beta_2^{s}&\cdots&\beta_{s}^{s}\\
		\end{pmatrix}\cdot\begin{pmatrix}
			\langle \rho,  E_1\rangle\\
			\langle \rho,  E_2\rangle\\
			\vdots\\
			\langle \rho,  E_s\rangle		
		\end{pmatrix}\\
		&= C_{\mathbb{E},\mathbb{B}}\Phi^{\mathbb{E}}(\rho).
	\end{align*}
	 where $C_{{\mathbb{E},\mathbb{B}}}$ is the change of basis matrix between $\mathbb{B}$ and $\mathbb{E}$. Therefore, since both basis of $\mathcal{B}$ are orthogonal $C_{{\mathbb{E},\mathbb{B}}}$ is an $s\times s$ unitary matrix.
\end{remark}

 \begin{lemma} 
 \label{lema properties of ms}
 Let $S$ be a subspace of $\C^n$ with $\dim(S)=r>0$, $\mathbb{B}=\{B_k\}_{i=1}^{\dim(\calB)}$ be a fixed orthonormal basis of Hermitian matrices of a C$^*$-subalgebra $\calB\subset M_n(\C)$ and  $\Phi^{\mathbb{B}}:\calD_{S}\to\R^{\dim(\calB)}$ defined as in \eqref{defi phi}. Then,
  \begin{enumerate}
 	\item $m_S\subset W\left(\{B_k\}_{k=1}^{\dim(\calB)} \right) $;
 	\item if there is $x\in S$ such that $Bx\neq 0$ for some $B\in\calB$, then  $(0,\dots,0)\notin m_S$. In particular, if $I\in \calB$ follows that $(0,\dots,0)\notin m_S$;
 	\item \label{mS convexo compacto} $m_S$ is a convex and compact set of $\R^{\dim(\calB)}$.
 \end{enumerate}
\end{lemma}
\begin{proof}
	\begin{enumerate}  
	\item If $x\in m_S$, then $
	x\in \text{conv}\left\{ \left(\left(v_i v_i^*\right)_1,\dots, \left(v_i v_i^*\right)_{\dim(\calB)} \right)  \right\}_{i=1}^{r}=
	\text{conv}\left( \{w_i\}_{i=1}^r\right)
	$	
	for some orthonormal basis $\{v_i\}_{i=1}^r$ of $S$ with 
	$$w_i=\left(\left(v_i v_i^*\right)_1,\dots, \left(v_i v_i^*\right)_{\dim(\calB)} \right).$$ 
	We also have that 
	$
	\left(v_i v_i^*\right)_k=\tr(v_i v_i^* B_k)=\langle B_kv_i,v_i\rangle\in \R,\ \text{for } k=1,\dots,\dim(\mathcal{B}), 
	$
	and hence
	\begin{equation} \label{wi}
		w_i= \left( \tr(v_i v_i^* B_1),\tr(v_i v_i^* B_2),...,\tr(v_i v_i^* B_{\dim(\mathcal{B})})\right)\in \R^{\dim(\mathcal{B})},
	\end{equation}
	for  $i=1,\dots, r$. Then, since $\tr(v_i v_i^*)=1$, follows that $w_i\in  W\left(\{B_k\}_{k=1}^{\dim(\mathcal{B})} \right)$ for $i=1,\dots,\dim(\mathbb{B})$, 
	and $x\in \text{conv}\{w_i: 1\leq i\leq r\}\subset W\left(\{B_k\}_{k=1}^{\dim(\mathcal{B})} \right) $.
	Therefore $m_S\subset W\left(\{B_k\}_{k=1}^{\dim(\mathcal{B})} \right)$.
	
	\item Using an orthonormal basis of $S$ containing $\frac1{\|x\|}x\in S$ and considering Remark \ref{change of basis}, any orthonormal basis of $\calB$ satisfies that $w_i\neq 0$ for at least one $i$ (with $w_i$ as in \eqref{wi}). Therefore each convex combination of  $\{w_i\}_{i=1}^r$ cannot be zero. Therefore, $(0,\dots,0)\notin m_S$.
	
	\item Note that $\Phi^{\mathbb{B}}$ is the restriction of a linear map from $M_n^h(\C)$ to $\R^n$, and that $\calD_S$ is a compact and convex set, hence $\Phi^{\mathbb{B}}(\calD_S)$ is also compact and convex.   In the next Proposition \ref{prop 1} it is proved that $\Phi^{\mathbb{B}}(\calD_{S})=m_S$ which completes the proof.
\end{enumerate}
\end{proof}

\begin{proposition}[Equivalent formulas for $m_S$] \label{prop 1}
 	Let $S$ be a subspace of $\C^n$ with $\dim(S)=r>0$, $\mathbb{B}=\{B_k\}_{i=1}^{\dim(\calB)}$ be a fixed orthonormal basis of Hermitian matrices of a C$^*$-subalgebra $\calB\subset M_n(\C)$ and  $\Phi^{\mathbb{B}}:\calD_{S}\to\R^{\dim(\calB)}$ defined as in \eqref{defi phi}. Then the following statements are equivalent to the Definition  \ref{defi mS} of $m_{S,\mathbb{B}}$ ($=m_S)$.
 %
%
	\begin{enumerate}
			\item \label{item3Prop1}
		$\displaystyle{
			m_S=\mathop{\bigcup}\limits_{\substack{\text{o.n. basis}  \\
					\{v_i\}_{i=1}^{r} \text{of } S}}
			{\rm conv}\left\{ \left(\left(v_i v_i^*\right)_1,\dots, \left(v_i v_i^*\right)_{\dim(\mathbb{B})} \right)  \right\}_{i=1}^{r}
		}$
		\item \label{mS = diag DsubS} $m_S= \Phi^{\mathbb{B}}(\calD_{S})$.
		\item \label{item2Prop1} 		
		\begin{equation*}
			\begin{split}
				m_S=\text{conv}\left(\left\{(V_1,V_2,\right.\right.&\left.\left.\dots,V_{\dim(\calB)})\right.\right.\in\R^{\dim(\calB)}:\\
				&\left.\left.V=s s^*, \text{ for } s\in S, \|s\|=1  \right\} \right)
			\end{split}
		\end{equation*}
		or equivalently
		\begin{equation}\label{eq momento}
			\begin{split}
			m_S=\text{conv}\left\{\left(\langle B_1v,v\rangle, \cdots, \langle B_{dim \calB}v,v\rangle\right)
			\right. 
			& \in\R^{\dim(\calB)}: \\
			 & \left. v\in S, \|v\|=1 \right\}.		
			\end{split}
		\end{equation}
	
		\item \label{item4Prop1}
		$\displaystyle{
		m_S=\{(\tr(B_1Y),\dots,\tr(B_{\dim(\mathbb{B})} Y)): Y\in \mathcal{D}_S\} \ \ \text{(see \eqref{def matrices de densidad de S})}
			}$
		\item \label{item5Prop1}
		$
		\displaystyle{
		W(P_SB_1P_S,\dots,P_SB_{_{\dim(\calB)}}P_S)=\bigcup_{\varepsilon\in[0,1]} \varepsilon m_S}$, (\text{see} \eqref{def jnr})
		
		\item \label{item6Prop1}
		$\displaystyle{
		m_S=W(P_SB_1P_S,\dots,P_SB_{{\dim(\mathbb{B})}}P_S)\cap \left\{ x\in \mathbb{R}^{{\dim(\mathbb{B})}}: \sum_{i=1}^{\dim(\mathbb{B})} x_i=1 \right\}
		}$.
	\end{enumerate}
\end{proposition}
\begin{proof}
\begin{enumerate}

	\item[\ref{item3Prop1}.] Is \eqref{def momento de un Subesp en relacion a B} of Definition \ref{defi mS}.
	
	\item[\ref{mS = diag DsubS} $\Leftrightarrow$ \ref{item3Prop1}.] Let $\rho\in\calD_{S}$. By \eqref{equ densidad escrita usando autovals y autovecs}, there exists an orthonormal basis $\{s_i\}_{i=1}^r$ of $S$ such that
	\begin{equation} \label{eq phi de v}
		\begin{split}
		\Phi^{\mathbb{B}}(\rho)&=\sum_{i=1}^{\dim(S)} \lambda_i \Phi^{\mathbb{B}}(s_i (s_i)^*)\\
		&=\sum_{i=1}^{\dim(S)} \lambda_i \left( (s_i (s_i)^*)_1, (s_i (s_i)^*)_2, \dots, (s_{i} (s_{i})^*)_{\dim(\calB)}\right), 
	\end{split}
	\end{equation}
	with $\sum_{i=1}^{\dim(S)}\lambda_i=1$ and $\lambda_i\geq 0$. Observe that, since $s_i s_i^*\geq 0$, $\tr(s_i s_i^*)=1$ and
	\begin{equation*}
		\begin{split}
	\Phi^{\mathbb{B}}(s_i (s_i)^*)&=\left(\tr(s_is_i^*B_1), \tr(s_is_i^*B_2),..., \tr(s_is_i^*B_{\dim(\mathcal{B})}) \right)\\
	&=\left(\langle B_1s_i,s_i\rangle,\langle B_2s_i,s_i\rangle,...,\langle B_{\dim(\mathcal{B})}s_i,s_i\rangle \right),		
		\end{split}
	\end{equation*}
	then each $\Phi^{\mathbb{B}}(s_i (s_i)^*)\in W\left(\{B_k\}_{k=1}^{\dim(\mathcal{B})} \right)$.
	Then, $\Phi^{\mathbb{B}}(\rho)\in\text{conv}\{\Phi^{\mathbb{B}}(s_i (s_i)^*)\}_{i=1}^{\dim(\mathcal{B})}\Rightarrow \Phi^{\mathbb{B}}(\rho)\in m_S$ which implies that $\Phi^{\mathbb{B}}(\calD_{S})\subseteq m_S$.
		
	In order to prove the other inclusion, take $x\in m_S$. Then there exist  $\{w_i\}_{i=1}^{r}$ as in \eqref{wi} such that
	 $x=\sum_{i=1}^{r}\alpha_iw_i$, with $0\leq \alpha_i$ for all $i$ and $\sum_{i=1}^{r}\alpha_i=1$. Then, 
	 \begin{align*}
	 	x&= \sum_{i=1}^{r}\alpha_i\left( \tr(v_i v_i^* B_1),\tr(v_i v_i^* B_2),...,\tr(v_i v_i^* B_{\dim(\mathcal{B})})\right) \\
	 	&=  \left( \sum_{i=1}^{r}\alpha_i\tr(v_i v_i^* B_1),\sum_{i=1}^{r}\alpha_i\tr(v_i v_i^* B_2),...,\sum_{i=1}^{r}\alpha_i\tr(v_i v_i^* B_{\dim(\mathcal{B})})\right) \\
	 	&= \left(\tr\left( \left( \sum_{i=1}^{r}\alpha_iv_i v_i^*\right)  B_1\right) 
	 	,...,\tr\left( \left( \sum_{i=1}^{r}\alpha_iv_i v_i^*\right)  B_{\dim(\mathcal{B})}\right) \right) \\
	 	&= \left(\tr\left( \rho B_1\right) ,\tr\left(\rho  B_2\right) ,...,\tr\left( \rho B_{\dim(\mathcal{B})}\right) \right)
	 \end{align*}
	 for $\rho=\sum_{i=1}^{r}\alpha_iv_i v_i^*$.	 
	 Since $\{v_i\}_{i=1}^r\subset S$ is an orthonormal basis of $S$, follows that $\rho\geq 0$, 
	 $\tr(\rho)=1$ and then $\rho \in \mathcal{D}_S$. Therefore $x= \Phi(\rho)\in \Phi(\mathcal{D}_S)$.	 
	 
	\item[\ref{item2Prop1} $\Leftrightarrow$ \ref{item3Prop1}.] Let $x\in m_S$. By the previous item there exist  $\rho\in\calD_{S}$ such that $x=\Phi(\rho)$. Observing \eqref{eq phi de v}, we deduce that 
	\begin{equation*}
		\begin{split}
	\Phi^\mathbb{B}(\rho)\in \text{conv}  \left\{ (V_1,V_2,\dots,V_{\dim(\calB)})\in\R^{\dim(\calB)}:   
	V=s s^*, \text{ for }  
	s\in S,    
	  	    \|s\|=1  \right\}  ,		
		\end{split}
	\end{equation*}
	where $V_k=\langle V, B_k\rangle_{tr}=\tr(V B_k)$ with $\mathbb{B}=\{B_k\}_{k=1}^{\dim \calB}$ the fixed basis of $\calB$.
	
	Now we prove the reverse inclusion. Let 
	\[y\in \text{conv} \left(\left\{(V_1, \dots,V_{\dim(\calB)})\in\R^{\dim(\calB)}:V=s s^*, \text{ for } s\in S, \|s\|=1  \right\} \right).\] 
	There exist $s_j\in S$, $\|s_j\|=1$ such that 
	$	y = \sum_{j=1}^{k}\beta_js_js_j^*$,
	with $0\leq \beta_j$ for all $1\leq j\leq k$ and $\sum_{j=1}^{k}\beta_j=1$. Then,
	\begin{align*}
		y& = \sum_{j=1}^{k}\beta_j\left( (s_js_j^*)_1,(s_js_j^*)_2,..., (s_js_j^*)_{\dim(\calB)}\right) \\
		& = \sum_{j=1}^{k}\beta_j\left( \tr(s_js_j^*B_1),\tr(s_js_j^*B_2),..., \tr(s_js_j^*B_{\dim(\calB)})\right),\\
	\end{align*}
		where $\left( \tr(s_js_j^*B_1),\tr(s_js_j^*B_2),..., \tr(s_js_j^*B_{\dim(\calB)})\right)\in m_S$ for every $1\leq j\leq k$. Therefore, $y\in m_S$.
	\end{enumerate}
 
	\begin{enumerate}
		\item[\ref{item4Prop1} $\Rightarrow$ \ref{item3Prop1}.]  If $Y\in \mathcal{D}_S$, then $Y=\sum_{i=1}^r \sigma_i\, y_iy_i^*$, with $\sum_{i=1}^r\sigma_i=1$, $\sigma_i\geq 0$, and $y_i\in S$ with $y_i\perp y_j$ for $i\neq j$ ($\{y_i\}_{i=1}^r$ is an orthonormal basis of $S$). Then 
	\begin{equation*}\begin{split}
			(\tr(B_1 Y),\dots,\tr(B_{\dim(\mathbb{B})} Y))&=\sum_{i=1}^r\sigma_i \left(\tr(B_1y_iy_i^*),\dots ,\tr(B_{\dim(\mathbb{B})} y_iy_i^*)\right)\\
			&=	\sum_{i=1}^r\sigma_i \left((y_iy_i^*)_1,\dots ,(y_iy_i^*)_{\dim(\mathbb{B})}\right)
		\end{split}
	\end{equation*}
	which belongs to $\mathop{\bigcup}\limits_{\substack{\text{o.n. basis}  \\
			\{v_i\}_{i=1}^{r} \text{of } S}}
	\text{conv}\left\{ \left(\left(v_i v_i^*\right)_1,\dots, \left(v_i v_i^*\right)_{{\dim(\mathbb{B})}} \right)  \right\}_{i=1}^{r}$.
	
	\item[\ref{item3Prop1} $\Rightarrow$ \ref{item4Prop1}.] Given a convex combination $ \sum_{i=1}^r \sigma_i\ \left(\tr(B_1 y_iy_i^*),\dots,\tr(B_{\dim(\mathbb{B})} y_iy_i^*) \right)$, where $\{y_i\}_{i=1}^r$ is  an orthonormal basis of $S$, $\sum_{i=1}^r\sigma_i=1$ with $\sigma_i\geq 0$, it is easy to see that $Y=\sum_{i=1}^r \sigma_i \, y_iy_i^*\in\DS$.

	\item[\ref{item4Prop1} $\Rightarrow$ \ref{item5Prop1}.] 
	Recall that
	\begin{equation*}
		\begin{split}
			W&(P_SB_1P_S,\dots,P_SB_{{\dim(\mathbb{B})}}P_S)=\\
			&=\{\left(\tr(\rho P_SB_1P_S), \dots, \tr(\rho P_SB_{\dim(\mathbb{B})}P_S)\right)\in\R^{\dim(\mathbb{B})}:
			 \\&  
			\qquad \rho\in M_n(\C),\rho\geq 0, \tr(\rho)=1\}.
		\end{split}
	\end{equation*} 
	Also observe that choosing $\rho=\frac1{\dim(S^\perp)}P_{S^\perp}$ follows that $(0,\dots,0)\in W(P_SB_1P_S,\dots,P_SB_{{\dim(\mathbb{B})}}P_S)$. Now for $\rho\in\mathcal{D}$, if $\tr(P_S\rho P_S)\neq 0$ we can write
	\begin{equation*}
		\begin{split}
			(\tr(\rho P_SB_1 P_S),\dots,&\tr(\rho P_SB_{\dim(\mathbb{B})} P_S))=\\
			&=(\tr(P_S\rho P_SB_1 ),\dots,\tr(P_S\rho P_SB_{\dim(\mathbb{B})}))\\
			&=
			\tr(P_S\rho P_S)(\tr(\rho_S B_1 ),\dots,\tr(\rho_S B_{\dim(\mathbb{B})}))		
		\end{split}
	\end{equation*}
	where $\rho_S=\frac1{\tr(P_S\rho P_S)}P_S\rho P_S\in \calD_S$. Hence, since $\tr(P_S\rho P_S)\leq 1$, we proved that every element of $W(P_SB_1P_S,\dots,P_SB_{{\dim(\mathbb{B})}}P_S)$ can be written as $\varepsilon(\tr(\rho_S B_1 ),\dots,\tr(\rho_S B_{\dim(\mathbb{B})}))$ with $\varepsilon\in[0,1]$ and $\rho_S\in \calD_S$. Then, since $(\tr(\rho_S B_1 ),\dots,\tr(\rho_S B_{\dim(\mathbb{B})}))\in m_S$, we have proved the inclusion $\subset$ of the equality in \ref{item5Prop1}.
	
	Now we will consider the inclusion $\supset$ of \ref{item5Prop1}. Suppose $C=\varepsilon m$, with $\varepsilon\in[0,1]$ and $m\in m_S$. Then $C=\varepsilon (\tr(B_1 Y),\dots ,\tr(B_{\dim(\mathbb{B})} Y))$ with $Y\in\calD_S$. The case when $\varepsilon=0$ is trivial. Now if $\varepsilon\in(0,1]$, choose $Z= \frac{1-\varepsilon}{\varepsilon} \frac1{\dim S^\perp}P_{S^\perp}$ and consider $\rho=\varepsilon(Y+Z)$. Then $\tr(\rho)=\varepsilon (\tr(Y)+\tr(Z))=\varepsilon (1+\frac{1-\varepsilon}{\varepsilon})=1$, $\rho\geq 0$ which implies that $\rho\in\calD$. Also observe that
	\begin{equation*}
		\begin{split}
			\varepsilon(\tr(B_1 Y),\dots ,\tr(B_{\dim(\mathbb{B})} Y))&=
			\varepsilon(\tr(B_1 P_S Y P_S),\dots ,\tr(B_{\dim(\mathbb{B})} P_SY P_S)) \\
			&=\varepsilon
			(\tr(B_1 P_S (Y+Z) P_S),\dots ,\\
			& \hspace{3cm} \tr(B_{\dim(\mathbb{B})} P_S(Y+Z) P_S))\\
			&=(\tr(B_1 P_S \rho P_S),\dots ,\tr(B_{\dim(\mathbb{B})} P_S\rho P_S))
			\end{split}
	\end{equation*}
	which belongs to $W(P_SB_1P_S,\dots,P_SB_{{\dim(\mathbb{B})}}P_S)$.

	\item[\ref{item5Prop1} $\Rightarrow$ \ref{item4Prop1}.] This implication follows considering $\varepsilon=1$ and $\rho=Y\in\calD_S$ in \ref{item5Prop1}.
	
	\item[\ref{item6Prop1}.] This statement can be proved using items \ref{mS = diag DsubS} and \ref{item5Prop1} of this proposition.
\end{enumerate}
\end{proof}

\begin{remark}
Items \ref{item5Prop1} and \ref{item6Prop1} of Proposition \ref{prop 1} prove that $m_S$ is a face of $W(P_SB_1P_S,\dots,P_SB_{{\dim(\mathbb{B})}}P_S)$.
\end{remark}

Observe that a change of basis affects the moment only up to a unitary transformation, as will be made precise in the next proposition.

\begin{proposition}[Change of basis in $\calB$] \label{propo: change of basis moment}
	Let $S$ be a subspace of $\C^n$, $\mathbb{B}=\{B_k\}_{k=1}^{\dim(\calB)}$ and $\mathbb{E}=\{E_k\}_{k=1}^{\dim(\calB)}$, be two different orthonormal basis of Hermitian matrices of a $C^*$-subalgebra $\calB\subset M_n(\C)$. Then,
	\begin{equation}\label{change of basis moment}
		m_S^{\mathbb{B}}= C_{\mathbb{E},\mathbb{B}}m_S^{\mathbb{E}}
	\end{equation}
	and
	\begin{equation}\label{change of basis jnr}
		W_{S,\mathbb{B}}=C_{\mathbb{E},\mathbb{B}}W_{S,\mathbb{E}}
	\end{equation}
	where $C_{\mathbb{E},\mathbb{B}}$ is the unitary matrix of $\dim(\calB)\times\dim(\calB)$ defined in Remark \ref{change of basis}. 
\end{proposition}
\begin{proof}
	Combining Remark \ref{change of basis} and item 1 in Proposition \ref{prop 1}  follows that $$m_S^{\mathbb{B}}=\Phi^{\mathbb{B}}(\calD_S)= C_{\mathbb{E},\mathbb{B}}\Phi^{\mathbb{E}}(\calD_S)=C_{\mathbb{E},\mathbb{B}}m_S^{\mathbb{E}}.$$
	Additionally, \eqref{change of basis jnr} is a direct consequence of \eqref{change of basis moment} and item 5 of Proposition \ref{prop 1}, since
	\begin{align*}
		W_{S,\mathbb{B}}&= \{\varepsilon x': x'\in m_S^{\mathbb{B}}, \varepsilon\in[0,1]\}\\
		&= \{\varepsilon C_{\mathbb{E},\mathbb{B}}x: x\in m_S^{\mathbb{E}}, \varepsilon\in[0,1]\}\\
		&= C_{\mathbb{E},\mathbb{B}}W_{S,\mathbb{E}}.
	\end{align*}	
\end{proof}

\begin{theorem} \label{teo equiv JNR y minimal}
Let $A\in M_n^h(\C)$, with $\pm\|A\|\in\sigma(A)$, $\calB$ a C$^*$-subalgebra of $M_n(\C)$ with the same unit $I$, $P_+$, $P_-$ the orthogonal projections corresponding to the eigenspaces $E_+$, $E_-$ of the eigenvalues $+\|A\|$ and $-\|A\|$ respectively, $\calD_{E_+}=\{\rho\in \calD:\  \ran(\rho)\subset E_+\}$, $\calD_{E_-}=\{\rho\in \calD: \ \ran(\rho)\subset E_-\}$ and	$\mathbb{B}=\{B_k\}_{k=1}^{\dim\calB}$ a Hermitian orthonormal basis of $\calB$ with the inner product defined by the trace.

The following statements are equivalent
\begin{enumerate}[label=(\roman*),ref=\textit{(\roman*)}] 
	\item \label{item (1) teo 2} $W\left(\{P_+ B_k P_+\}_{k=1}^{\dim\calB}\right)\cap W\left(\{P_- B_k P_-\}_{k=1}^{\dim\calB}\right)\neq \{(0,\dots,0)\}$.
	\item \label{item (2) teo 2}  There exist $\rho_+\in \calD_{E_+}$ and $\rho_-\in \calD_{E_-}$ such that $\tr(\rho_+B_k )=\tr(\rho_-B_k )$, for all $k=1,\dots,\dim\calB$.
	
	\item \label{item (3) teo 2} $m_{E_+}\cap m_{E_-}\neq \emptyset$.
	\item \label{item (4) teo 2} 
	There exist a Hermitian $X=\left(\begin{smallmatrix}
		X_{1,1}&0&0\\ 0&X_{2,2}&0\\0&0&0
	\end{smallmatrix}\right)$ (in the orthogonal decomposition used in \eqref{equ A y X en descomp ortog de A})  and $X\in\calB_\tr^\perp\setminus\{0\}$ such that $A X =\|A\| \left| X\right|$.

	\item \label{item (5) teo 2} $A=\|A\|(P_+-P_-)+R$ with $R\in M_n^h(\C)$,   $\ran(R)\subset \ran(I-P_+-P_-)$ and $m_{E+}\cap m_{E_-}\neq \emptyset$.
	\item \label{item (6) teo 2} $A$ is $\calB$-minimal
\end{enumerate}
\end{theorem}
\begin{proof}
	\ref{item (1) teo 2} $\Rightarrow$ \ref{item (2) teo 2}  In this case there exists  $0\leq\rho, \mu\in M_n(\C)$ with $\tr(\rho)=\tr(\mu)=1$ that satisfy $\tr(P_+\rho P_+ B_k )=\tr(P_-\mu P_- B_k ) $ for every $k=1,\dots, \dim(\calB)$ and at least a $k_0$ such that $\tr(P_+\rho P_+ B_{k_0})\neq 0$. 
	Then, since $I=\sum_{k=1}^{\dim \calB} \tr(I B_k ) B_k$, follows that 
	\begin{equation*}
	\begin{split}
	  \tr(P_+\rho P_+ )&= \tr\left(P_+\rho P_+  \sum_{k=1}^{\dim \calB} \tr(B_k ) B_k\right)=\sum_{k=1}^{\dim \calB} \tr(B_k )\ \tr\left(P_+\rho P_+   B_k\right)\\
	&=\sum_{k=1}^{\dim \calB} \tr(B_k )\ \tr\left(P_-\mu P_-   B_k\right)=
	 \tr\left(P_-\mu P_-  \sum_{k=1}^r \tr(B_k ) B_k\right)\\
	 &=\tr\left(P_-\mu P_- \right)
	\end{split}
	\end{equation*}
	Now suppose that $\tr\left(P_+\rho P_+\right)=0$, then since $P_+\rho P_+\geq 0$ follows that $P_+\rho P_+=0$, and hence $P_+\rho P_+B_{k_0}=0$, a contradiction. Therefore $\tr\left(P_+\rho P_+\right)>0$ holds and we can define
	\begin{equation*}
		\begin{split}
			\rho_+ =\frac{P_+\rho P_+}{\tr\left(P_+\rho P_+\right)}\ \in \calD_{E_+}\ \ and \  
			\mu_- =\frac{P_-\mu P_-}{\tr\left(P_-\mu P_-\right)}\ \in \calD_{E_-}
		\end{split}
	\end{equation*}
	which clearly satisfy that $\tr(\rho_+  B_k )=\tr(\mu_-B_k)$ for all $k=1,\dots,\dim\calB$. These last equalities clearly imply that statement \ref{item (2) teo 2} holds.
	
 \medspace
 
 	\ref{item (2) teo 2} $\Rightarrow$ \ref{item (1) teo 2} Is trivial choosing $\rho=\rho_+$ and $\mu=\rho_-$.
 
 \medspace

\ref{item (2) teo 2} $\Leftrightarrow$ \ref{item (3) teo 2} Follows directly using \ref{mS = diag DsubS} of Proposition \ref{prop 1}.

 \medspace

\ref{item (4) teo 2} $\Rightarrow$ \ref{item (2) teo 2} Since $X_{1,1}$, $X_{2,2}$ are not null we can suppose that $X=\left(\begin{smallmatrix}
	X_{1,1}&0&0\\ 0&X_{2,2}&0\\0&0&0
\end{smallmatrix}\right)$ has $\|X\|_1=1$. Then choosing $\rho_+=\frac1{\tr(X_{1,1})} X_{1,1}$ and $\rho_-=\frac1{\tr(X_{2,2})} X_{2,2}$ (see \eqref{equ A y X en descomp ortog de A}) follows that $\rho_+\in\calD_{E_+}$ and $\rho_-\in\calD_{E_-}$ and \ref{item (2) teo 2} holds. 

\ref{item (2) teo 2} $\Rightarrow$ \ref{item (4) teo 2}  Using $\rho_+$ and $\rho_-$ from \ref{item (2) teo 2} we can construct $X=\left(\begin{smallmatrix}
	\rho_+&0&0\\ 0&-\rho_-&0\\0&0&0
\end{smallmatrix}\right)$ that satisfies \ref{item (4) teo 2}.

\ref{item (4) teo 2} $\Rightarrow$ \ref{item (5) teo 2} 
We can write $X=\sum_{i=1}^l\alpha_i v_i (v_i)^*+\sum_{j=1}^m\beta_j w_j(w_j)^*$, with $v_i\in E_+$, $w_j\in E_-$ norm one eigenvectors of $X$, and $\alpha_i\in[0,1]$, $\beta_j\in[-1,0]$ its respective eigenvalues with and $\sum_{i=1}^l\alpha_i +\sum_{j=1}^m\beta_j=0$. If we take $X$ such that $\|X\|_1=1$, then $\sum_{i=1}^l\alpha_i +\sum_{j=1}^m(-\beta_j)=1$. Therefore, since $X\in \calB_{\tr}$, $0=\langle X,B_k\rangle=\sum_{i=1}^l\alpha_i \left(v_i (v_i)^*\right)_k+\sum_{j=1}^m\beta_j\left( w_j(w_j)^*\right)_k$, for every $k=1,\dots,\dim \calB$, and hence
\begin{equation}
	\label{eq igualdad de elem de momentos}
	\begin{split}
		\sum_{i=1}^l 2\alpha_i & \left(\left(v_i (v_i)^*\right)_1,\dots, \left(v_i (v_i)^*\right)_{\dim \calB}\right)=\\
		&=\sum_{j=1}^m(-2\beta_j)\left(\left( w_j(w_j)^*\right)_1,\dots,\left( w_j(w_j)^*\right)_{\dim\calB}\right)
		\end{split}
\end{equation}
which proves that $m_{E+}\cap m_{E_-}\neq \emptyset$ and then \ref{item (5) teo 2} holds.

\ref{item (5) teo 2} $\Rightarrow$ \ref{item (4) teo 2} If there is an element  $\varepsilon\in m_{E+}\cap m_{E_-}$, then there exist orthogonal elements $\{v_i\}_{i=1}^l\subset E_+$ and 
$\{w_j\}_{j=1}^m\subset E_-$, $\alpha_j$ and $\beta_j$, such that \eqref{eq igualdad de elem de momentos} holds and it is equal to $\varepsilon$. This allows the construction of $X=\sum_{i=1}^l\alpha_i v_i (v_i)^*+\sum_{j=1}^m\beta_j w_j(w_j)^*$ that satisfies \ref{item (4) teo 2}.

\ref{item (4) teo 2} $\Leftrightarrow$ \ref{item (6) teo 2} Is Theorem \ref{teo X mod}.
 \end{proof}

	\begin{corollary}\label{coro: equival def soporte}
		Let $V, W$ be two non trivial orthogonal subspaces of $\C^n$, $\mathcal{B}$ a C$^*$-subalgebra of $M_n(\C$) with unit $I=I_n$ and $\mathbb{B}=\{B_k\}_{k=1}^{\dim(\calB)}$ an orthogonal basis of $\calB$,
		then the following statements are equivalent.
		\begin{enumerate}
			\item $(V,W)$ is a support, that is, $m_V\cap m_W\neq \emptyset$ (see \cite{soportes}).
			\item \label{item 2 corol soportes} The Hermitian matrix $M=\lambda (P_V-P_W)+R$ is $\mathcal{B}$-minimal  for every $\lambda\in\R_{>0}$, $R\in M_n^h(\C)$, $\|R\|\leq \lambda$, $R(P_V +P_W)=0$.
 
			\item The sets 	
			$
			\mathcal{D}_V=\{\rho\in M_n^h(\C):P_V \rho=\rho\geq 0,\ \tr(\rho)=1\}$ and $\mathcal{D}_W=\{\mu\in M_n^h(\C): P_W \mu=\mu\geq 0,\ \tr(\mu)=1\}$ (see \eqref{def matrices de densidad de S}), 
			satisfy  
			$$
			\Phi^\mathbb{B}(\calD_V)\cap\Phi^{\mathbb{B}}(\calD_W)\neq\emptyset .
			$$
		\end{enumerate}
	\end{corollary}
\begin{proof} The equivalence between 1. and 2. follows from the equivalence of the items \ref{item (3) teo 2} and \ref{item (5) teo 2} in Theorem \ref{teo equiv JNR y minimal}, and between 1. and 3. is obtained from \ref{mS = diag DsubS}. of Proposition \ref{prop 1}.
\end{proof}

  \begin{remark}
  The previous result and Theorem~\ref{teo equiv JNR y minimal} prove that, when $I\in\calB$, all $\calB$-minimal matrices are of the form given in item~\ref{item 2 corol soportes} of Corollary~ \ref{coro: equival def soporte}. 	
  \end{remark}
 
\section{Relations between subdifferentials and moment} \label{section subdiff and moment}
In this section, we follow the main ideas of \cite{overton2} and \cite{BV-SDP} to characterize the moment of the eigenspace of the maximum eigenvalue of a minimal matrix in terms of the subdifferentials of the same eigenvalue and the operator norm.

First, we introduce the definition of subdifferential of a convex function, as was given in \cite{bata-grover}.

\begin{definition}\label{def subdiferencial}
	For any convex function $f:\mathcal{X}\to \R$ defined on a Banach space $\mathcal{X}$, and $\mathcal{X}^*$ its dual, it can be defined the subdifferential of $f$ at $x\in\mathcal{X}$ as
	\begin{equation}\label{subdif gral banach}
		\partial f(x)=\{v\in\mathcal{X}^*: f(y)-f(x)\geq \text{Re } v(y-x), \forall\ y\in \mathcal{X}\}, 
	\end{equation}
\end{definition}
In particular, for the Hermitian matrix space $\mathcal{X}=M_n^{h}(\C)$, the subdifferential of $f(\cdot)=\lambda_{max}(\cdot)$ at $x=A\in M_n^h(\C)$ is
\begin{align*}
	\partial \lambda_{max}(A)=& \left\lbrace V\in M_n^{h}(\C): \lambda_{max}(Y)-\lambda_{max}(A)\geq \text{Re}\left(  \tr(V(Y-A))\right),\right. \\ 
	&\left. \text{for every } Y\in M_n^{h}(\C) \right\rbrace .
\end{align*}
A more practical (or useful) formula for the subdifferential of $\lambda_{max}$ at $A$, proved first for symmetric real matrices in \cite{overton2} and later extended por Hermitian in \cite{BV-SDP}, is
\begin{eqnarray} \label{subdif matrices}
	\partial\lambda_{max}(A)&=&\text{co}\{qq^*: Aq=\lambda_{max}(A)q \text{ and }\|q\|=1\}\\ \nonumber
	&=&\{Q_{max}R_sQ_{max}^*: R_s\in M_s^{h}(\C), R_s\geq 0, \tr(R_s)=1\},
\end{eqnarray}
where the columns of $Q_{max}$ form an orthonormal set of $s$ eigenvectors for $\lambda_{max}(A)$ (an orthonormal basis of the eigenspace of $\lambda_{max}(A)$). Observe that $Q_{max}$ depends on the matrix $A$. 

From now we denote as $\calD^s$ the set of density matrices of $s\times s$.

Now consider $A(x)=A_0+\sum_{k=1}^{\dim{\calB}}x_kB_k$, with $A_0\in  M_n^{h}(\C)$, $\{B_k\}_{k=1}^{\dim{\calB}}$ be a fixed Hermitian basis of $\mathbb{B}^h$ and $x\in \R^{\dim{\calB}}$. The maximum eigenvalue of $A(x)$, $\lambda_{max}(A(x))$, is a map from $\R^{\dim{\calB}}$ to $\R$ and a composition between a smooth function $A(\cdot)$ and a convex map $\lambda_{max}(\cdot)$. For every $k$, the partial derivatives of $A$ at $x$ are
$$\frac{\partial A}{\partial x_k}(x)=B_k.$$
Adapting Theorem 3 in \cite{overton2} to the Hermitian case, the subdifferential of $\lambda_{max}(A(x))$ is
\begin{equation} \label{subdif A(x)}
	\begin{split}
	\partial\left(\lambda_{max}(A(x))\right)=\left\lbrace v\in \R^{\dim{\calB}}:\ v_k=\tr(R_s Q_{max}(x)^*B_kQ_{max}(x)),\right. \\
	\left. \text{with } R_s\in \calD^s\right\rbrace ,	
	\end{split}	
\end{equation}
where $s$ is the multiplicity of $\lambda_{max}(A(x))$ and the columns of $Q_{max}(x)$ form an orthonormal basis of eigenvectors for $\lambda_{max}(A(x))$. Observe that $Q_{max}(x)$ depends on $A(x)$ and $v \in W\left(\{B_k\}_{k=1}^{\dim(\calB)} \right)$ if $v\in\partial\left(\lambda_{max}(A(x))\right)$. Then,
\[\partial\left(\lambda_{max}(A(x))\right)\subseteq W\left(\{B_k\}_{k=1}^{\dim(\calB)} \right).\]
Since $\lambda_{max}(A(x))=\lambda_{max}\circ A(x)$, the generalized derivative and generalized gradient coincide with the directional derivative and subdifferential, respectively (see Proposition 2.2.7 in \cite{clarke1}). Then, the expression \eqref{subdif A(x)} is valid by the chain rule for subdifferentials, and is useful to obtain the following characterization of the subdifferential and the partial derivative of the maximum eigenvalue of $A(x)$.

\begin{theorem}\label{teo subdif=ms}
Let $A(x)=A_0+\sum_{k=1}^{\dim{\calB}}x_kB_k$, with $A_0\in  M_n^{h}(\C)$, $\{B_k\}_{k=1}^{\dim{\calB}}$ be a fixed Hermitian basis of $\mathcal{B}^h$, $x\in \R^{\dim{\calB}}$ and $S_{max}$ be the eigenspace of $\lambda_{max}(A(x))$. Then
\begin{equation}\label{subdiff l1 = momento}
			\partial\big(\lambda_{max}(A(x))\big)=\Phi\big(\partial\lambda_{max}(A(x))\big)=\Phi(\mathcal{D}_{S_{max}})= m_{S_{max}},
\end{equation}
where $m_{S_{max}}$ is the moment of the eigenspace $S_{max}$ and $\Phi:\calD_{S}\to\R^{\dim(\calB)}$ defined as in \eqref{defi phi}.	

Additionally, the directional derivative of $\lambda_{max}$ at $x\in \R^{\dim{\calB}}$ in the direction $w=(w_1,w_2,...,w_{\dim{\calB}})\in \R^{\dim{\calB}}$ that is defined by
\begin{equation} \label{partial der}
	\begin{split}
\lambda'_{max}(x,w)&=\lim\limits_{t\to 0^+}\dfrac{\lambda_{max}(A(x+tw))-\lambda_{max}(A(x))}{t}\\
&= \lambda_{max}\left(\sum_{k=1}^{\dim{\calB}}w_kQ_{max}(x)^*B_kQ_{max}(x) \right),
\end{split}
\end{equation}
where the columns of $Q_{max}(x)$ form an orthonormal basis of eigenvectors for $\lambda_{max}(A(x))$.
\end{theorem}

\begin{proof}
	Suppose $\lambda_{max}(A(x))$ has multiplicity $s$ and $S_{max}$ is the eigenspace of $\lambda_{max}(A(x))$ with a fixed orthonormal basis of eigenvectors $\{q_1(x),...,q_s(x)\}$. If $Q_{max}(x)=\begin{bmatrix}
		q_1(x)|...|q_s(x)
	\end{bmatrix}\in \C^{n\times s}$. Let $v\in\partial\lambda_{max}(A(x))$. As we observed previously,  $v\in W\left(\{B_k\}_{k=1}^{\dim(\calB)} \right)$ and  by \eqref{subdif A(x)},  has coordinates
	$$v_k=\tr \left( R_s Q_{max}(x)^*B_kQ_{max}(x)\right)=\tr \left(Y(x)B_k\right),$$ 
	with $R_s\in M_s^h(\C)$, $R_s\geq 0$, $\tr(R_s)=1$, for every $k=1,...,\dim{\calB}$. Observing that \[Y(x)=Q_{max}(x)R_s Q_{max}(x)^*\in\mathcal{D}_{S_{max}},\] 
	then
	 \[v=(v_1,v_2,...,v_n)=\Phi(Y(x))\in m_{S_{max}}\] 
	with $\Phi$ as in \eqref{defi phi}.
	
	On the other hand, take any $v\in m_{S_{max}}$.  The following argument is essentially the same as in Theorem 3.3 in \cite{BV-SDP}, but we include it for completeness. By Proposition \ref{prop 1}, it can be written as
	$$
	v=\left( \tr(B_1Y), \tr(B_2Y), ...,\tr(B_{\dim{\calB}}Y) \right), $$
	with $Y\in M_n^h(\C)$, $\tr(Y)=1$ and $Im(Y)\subset S_{max}$.
	In terms of the orthogonal decomposition $\C^n=S_{max}\oplus S_{max}^{\perp}$, given by the matrix $Q=\begin{bmatrix}
		Q_{max}(x)&Q_2(x)
	\end{bmatrix}$ ($Q_2$ is a matrix whose columns form an orthonormal set for $S_{max}^{\perp}$ and $Q$ is an unitary matrix), $Y$ is defined by
	$$Y=Q\begin{bmatrix}
		V&0\\
		0&0
	\end{bmatrix}Q^*=Q_{max}(x)VQ_{max}(x)^*$$
	with $V\in M_s^h(\C)$, $1=\tr(Y)=\tr(V)$ and $V\geq 0$. Therefore, $v\in \partial \lambda_{max}(A(x))$.
	
	The proof of the directional derivative is analogous to the case $B_k=e_ke_k^*$ made in Proposition 3.5 in \cite{BV-SDP}, since for every $w\in \R^n$ 
	\[\lambda_{max}'(x,w)=\max_{v\in \partial \lambda_{max}(A(x))} \left\langle v,w\right\rangle,\]
		\begin{eqnarray*}
		\lambda_{max}'(x,w)&=&\max\left\lbrace \sum_{k=1}^{\dim{\calB}} v_kw_k: v_k=\tr\left( R_s Q_{max}(x)^*B_kQ_{max}(x)\right): \text{ with } R_s\in \calD^s\right\rbrace\\
		&=&\max\left\lbrace  \tr\left(  R_s\sum_{k=1}^{\dim{\calB}}w_k Q_{max}(x)^*B_kQ_{max}(x)\right) :\ \text{with } R_s\in \calD^s\right\rbrace\\
		&=&\lambda_{max}(B(w)),
	\end{eqnarray*} 
	where the last equality  is due to the fact that $B(w)=\sum_{k=1}^{n}w_kQ_{max}(x)^*B_kQ_{max}(x)\in M_s^{h}(\C)$ and
	\begin{equation} \label{ray mat}
		\lambda_{max}(B(w))=\max \{\tr(B(w)R):\ R_s\in \calD^s\}. 
	\end{equation}

\end{proof}

We relate the subdifferential of the spectral norm of $A(x)$ to the subdifferentials of the minimum and maximum eigenvalues of $A(x)$, in a manner analogous to the one considered in the particular case $\mathcal{B}=\text{Diag}(M_n(\C))$ in \cite{BV-SDP}.

\begin{proposition} \label{eq subdif norma -eig matrix}
For $A(x)=A_0+\sum_{k=1}^{\dim{\calB}}x_kB_k$, with $A_0\in  M_n^{h}(\C)$, $\{B_k\}_{k=1}^{\dim{\calB}}$ be a fixed Hermitian basis of $\mathcal{B}^h$, $x\in \R^{\dim{\calB}}$, let $\lambda_{min}(A(x))$ and $\lambda_{max}(A(x))$ be the minimum and maximum eigenvalue of $A(x)$, respectively. Then,
\begin{enumerate}
	\item the subdifferential of the spectral norm of $A(x)$ is
		\begin{equation}\label{watsonhermit}
		\partial\|A(x)\|={\rm conv }\{uu^*: A(x)u=\|A(x)\|u \text{ and }\|u\|=1\}.
	\end{equation}
	\item If $S_{min}$ is the eigenspace corresponding to $\lambda_{min}(A(x))$, then
	\begin{equation}\label{subdiff min = momento}
		\partial\big(\lambda_{min}(A(x))\big)= -m_{S_{min}}.
	\end{equation}	
	\item We deduce that
	
	$\partial\|A(x)\|=$
	\begin{equation}
			\left\lbrace \begin{array}{lll}
				\partial\lambda_{max}(A(x))&\text{if}& \|A(x)\|=\lambda_{max}(A(x))\\
				\partial\lambda_{min}(A(x))&\text{if}& \|A(x)\|=-\lambda_{min}(A(x))=|\lambda_{min}(A(x))|\\
				{\rm conv}\left(\partial \lambda_{max}(A(x))\cup\partial \lambda_{min}(A(x)) \right)&\text{if}& \|A(x)\|=-\lambda_{min}(A(x))=\lambda_{max}(A(x)).
			\end{array}\right. 
	\end{equation}
	
\end{enumerate}	
\end{proposition}

\begin{proof}
	
\begin{enumerate}
	\item The proof of the subdifferential of the matrix norm of $A(x)$ is a direct consequence of the formulas obtained in \cite{watson} and \cite{bata-grover} for real  and complex cases, respectively.
	\item 	Since $\lambda_{min}(A(x))=-\lambda_{max}(-A(x))$ for any $A(x)\in M_n^{h}(\C)$, then by Theorem \ref{teo subdif=ms}
	\[\partial \lambda_{min}(A(x))=-\partial \lambda_{max}(-A(x))= -\Phi\left(\partial \lambda_{max}(-A(x)) \right)=-\Phi\left(\mathcal{D}_{S} \right) , \]
	with $\mathcal{D}_{S}=\text{conv}\{ss^*:\ \|s\|=1,\ -A(x)s=\lambda_{max}(-A(x))s\}$ ($S$ is the eigenspace associated to $\lambda_{max}(-A(x))$). Therefore
	\begin{equation*}
		\begin{split}
			\partial \lambda_{min}(A(x))&=-\partial \lambda_{max}(-A(x))\\
			&=-\Phi\left(\text{conv} \{ss^*: \|s\|=1,\ -A(x)s=\lambda_{max}(-A(x))s\} \right) \\
			&=-\Phi\left(\text{conv} \{ss^*: \|s\|=1,\ A(x)s=\lambda_{min}(A(x))s\} \right) \\
			&=-\Phi(\mathcal{D}_{S_n})=-m_{S_n}.
		\end{split}
	\end{equation*}
	
	\item It can be directly deduced from the formulas \eqref{watsonhermit}, \eqref{subdiff l1 = momento} and \eqref{subdiff min = momento} for each case.
\end{enumerate}

\end{proof}

\begin{theorem}\label{teo equiv subdif, momentos, etc para matrices}
	Let $A(x)=A_0+\sum_{k=1}^{\dim{\calB}}x_kB_k$, with $A_0\in  M_n^{h}(\C)$, $\{B_k\}_{k=1}^{\dim{\calB}}$ be a fixed Hermitian basis of $\mathcal{B}^h$, $x\in \R^{\dim{\calB}}$, and suppose that $\lambda_{max}(A(x))=-\lambda_{min}(A(x))$. Then, the following statements are equivalent
	
	\begin{enumerate}
		\item[(1)] $0\in \partial\|A(x)\|$.
		\item[(2)] $0\in\partial\lambda_{max}(A(x))+\partial\lambda_{min}(A(x))$.
		\item[(3)] $m_{S_{max}}\cap m_{S_{min}}\neq \emptyset$, where $S_{max}$ and $S_{min}$ are the eigenspaces of $\lambda_{max}(A(x))$ and $\lambda_{min}(A(x))$, respectively. 
		\item[(4)] $W\left( \left\lbrace P_{S_{max}}B_kP_{S_{max}}\right\rbrace_{k=1}^{\dim{\calB}} \right)\cap W\left( \left\lbrace P_{S_{min}}B_kP_{S_{min}}\right\rbrace_{k=1}^{\dim{\calB}} \right)\neq \{\overline{0}\} $. 
		\item[(5)] $A(x)$ is minimal.
	\end{enumerate}
\end{theorem}

\begin{proof}
	(1)$\Leftrightarrow$(3) If $0\in \partial \|A(x)\|$ and $\lambda_{max}(A(x))=-\lambda_{min}(A(x))$, then using Proposition \ref{eq subdif norma -eig matrix}
	$$0\in {\rm conv}\left(\partial \lambda_{max}(A(x))\cup\partial \lambda_{min}(A(x)) \right)={\rm conv}\left(m_{S_{max}}\cup -m_{S_{min}} \right)
	$$ 
	and there exist $\alpha\in (0,1)$, $Y_0\in \mathcal{D}_{S_{max}}$ and $Z_0\in \mathcal{D}_{S_{min}}$ such that $0=\alpha \Phi(Y_0)+(1-\alpha)\Phi(-Z_0)$. Both matrices $Y_0$ and $Z_0$ fulfill that $\tr(Y_0)=\tr(Z_0)=1$ (see \eqref{def matrices de densidad de S}). Thus, we obtain that $\alpha=\frac 12$ and then $\Phi(Y_0)=\Phi(Z_0)$. Therefore, $m_{S_{max}}\cap m_{S_{min}}\neq \emptyset$.

	The converse implication can be proved reversing the previous steps.

	(2)$\Leftrightarrow$(3) If we consider
	the formulas $\partial\big(\lambda_{max}(A(x))\big)= m_{S_{max}}$
	and $\partial \lambda_{min}(A(x))=-m_{S_{min}}$ from  \eqref{subdiff l1 = momento} and  \eqref{subdiff min = momento}, respectively, then it is trivial that   $m_{S_{max}}\cap m_{S_{min}}\neq \emptyset$ if and only if $0\in m_{S_{max}}-m_{S_{min}}=\partial \lambda_{max}(A(x))+\partial \lambda_{min}(A(x))$.
	
	The equivalence (4)$\Leftrightarrow$(5) was proved in Theorem \ref{teo equiv JNR y minimal}.
	
	(3)$\Leftrightarrow$(4) Using that $S_{max}=E_+$, $S_{min}=E_-$ in Theorem \ref{teo equiv JNR y minimal}, we obtain that 
	\[W\left( \left\lbrace P_{S_{max}}B_kP_{S_{max}}\right\rbrace_{k=1}^{\dim{\calB}} \right)\cap W\left( \left\lbrace P_{S_{min}}B_kP_{S_{min}}\right\rbrace_{k=1}^{\dim{\calB}} \right)\neq \{0\} \]
	is equivalent to there exist $\rho_+\in \calD_{S_{max}}$ and $\rho_-\in \calD_{S_{min}}$ such that $\tr(\rho_+B_k )=\tr(\rho_-B_k )$, $\forall k=1,\dots,\dim\calB$. Observe that the vector 
	\begin{equation*}
		\begin{split}
			v&=\left(\tr(\rho_+B_1), \tr(\rho_+B_2),...,\tr(\rho_+B_{\dim\calB}) \right) \\ &=\left(\tr(\rho_-B_1), \tr(\rho_-B_2),...,\tr(\rho_-B_{\dim\calB})\right) \in m_{S_{max}}\cap m_{S_{min}}.		
		\end{split}
	\end{equation*}
\end{proof}

\section{Cases and Examples}\label{sec examples}
In this section we illustrate some cases and examples of $\calB$-minimal matrices and moments of subspaces, including a case related to quantum information theory in Example \ref{quantum ex}.

Consider $\mathcal{B}_n= \{D\in M_n(\C): D\ \text{ is diagonal }\}$ the $C^*$-subalgebra of diagonal matrices respect an orthonormal fixed basis $\{e_k\}_{k=1}^n$ of $\C^n$ and a subspace $S\subset \C^n$. 
In \cite{KV moment} it was characterized the moment of $S$ respect  the orthonormal real basis in $\mathcal{B}_n$ given by
\begin{equation}
	\mathbb{E}=\{E_i= e_ie_i^*\in M_n(\C)\}_{i=1}^n.
\end{equation}
The moment of $S$ was defined in terms of this basis as
\begin{equation}
	m_S^{\mathbb{E}}= \phi^{\mathbb{E}}(\mathcal{D}_S),
\end{equation}
where $\phi^{\mathbb{E}}(Y)=\text{Diag}(Y)=\sum_{i=1}^{n}Y_{ii}E_i$. Also, the joint numerical range respect this family is
\begin{align*}
	W_{S,\mathbb{E}}&=\{(\tr(\rho P_SE_1P_S), \tr(\rho P_SE_2P_S),...,\tr(\rho P_SE_2P_S))\in \R^n: \rho \in \mathcal{D}\}\\
	&= \{\varepsilon x: x\in m_S^{\mathbb{E}}, \varepsilon\in[0,1]\}.
\end{align*}
Now consider another orthonormal real basis in $\mathcal{B}_n$, $\mathbb{B}=\{B_i\}_{i=1}^n$. By Remark \ref{change of basis}, the function $\phi_S^B$ can be described as
\begin{align*}
	\phi^{\mathbb{B}}(Y)&= (\tr(YB_1), \tr(YB_2),...,\tr(YB_n))\\
	& =\begin{pmatrix}
		\beta_1^1&\beta_2^1&\cdots&\beta_{n}^1\\
		\beta_1^2&\beta_2^2&\cdots&\beta_{n}^2\\
		\vdots&\vdots&\vdots&\vdots\\
		\beta_1^{n}&\beta_2^{n}&\cdots&\beta_{n}^{n}\\
	\end{pmatrix}\cdot\begin{pmatrix}
		\langle Y,  E_1\rangle\\
		\langle Y,  E_2\rangle\\
		\vdots\\
		\langle Y,  E_n\rangle		
	\end{pmatrix}\\
	&= C_{\mathbb{E},\mathbb{B}}\Phi^{\mathbb{E}}(Y)= C_{\mathbb{E},\mathbb{B}}\Diag(Y).	
\end{align*} 
where $\beta_i^k=\langle B_k,E_i\rangle=\tr(B_kE_i)= (B_k)_{ii}$ for every $1\leq k,i\leq n$. Then, the $k$-row of $C_{\mathbb{E},\mathbb{B}}$ are the diagonal of $B_k$, for $1\leq k\leq n$. Observe also that $\langle B_i,B_j\rangle=\delta_{ij}$, since $B$ is a diagonal (real) orthonormal basis, so $ C_{\mathbb{E},\mathbb{B}}^t=C_{\mathbb{B},\mathbb{E}}$. Then, by Proposition \ref{propo: change of basis moment} the moment of $S$ and the joint numerical range respect to this basis can be written as
\begin{equation}\label{momento ej cambio de base}
	m_S^{\mathbb{B}}= \begin{pmatrix}
		\text{---}&\Diag(B_1)&\text{---}\\
		\text{---}&\Diag(B_2)&\text{---}\\
		\vdots&\vdots&\vdots\\
		\text{---}&\Diag(B_n)&\text{---}
	\end{pmatrix}m_S^{\mathbb{E}}= \begin{pmatrix}
	\text{---}&\Diag(B_1)&\text{---}\\
	\text{---}&\Diag(B_2)&\text{---}\\
	\vdots&\vdots&\vdots\\
	\text{---}&\Diag(B_n)&\text{---}
	\end{pmatrix} \Diag(\calD_{S})
\end{equation}
and
\begin{equation}\label{jnr ej cambio de base}
	W_{S,\mathbb{B}}= \begin{pmatrix}
		\text{---}&\Diag(B_1)&\text{---}\\
		\text{---}&\Diag(B_2)&\text{---}\\
		\vdots&\vdots&\vdots\\
		\text{---}&\Diag(B_n)&\text{---}
	\end{pmatrix}W_{S,\mathbb{E}},
\end{equation}
respectively.

\begin{example}
	Consider the subspace $S\subseteq\C^3$ given by $S=\text{span}\{\frac{1}{\sqrt{2}}(1,1,0);(0,0,1)\}$. Define the rank one orthogonal projections
	\begin{align*}
		P&=\frac{1}{\sqrt{2}}(1,1,0)\otimes \frac{1}{\sqrt{2}}(1,1,0)=\begin{pmatrix}
			\frac 12&\frac 12&0\\
			\frac 12&\frac 12&0\\
			0&0&0
		\end{pmatrix},\\		
		Q&=(0,0,1)\otimes (0,0,1)=\begin{pmatrix}
			0&0&0\\
			0&0&0\\
			0&0&1
		\end{pmatrix}.\\
	\end{align*}
	Using these projections, the set $\calD_{S}$ in this case is
	\[\calD_{S}=\{Y\in \calD:\ Y=\alpha P+(1-\alpha)Q,\ 0\leq \alpha\leq 1\}.\]
	Then, the moment in the basis $\mathbb{E}=\{E_1;E_2;E_3\}$ is
	\[m_S^{\mathbb{E}}=\Diag(\calD_{S})= \left\lbrace \left( \frac{\alpha}{2}, \frac{\alpha}{2},1-\alpha\right):\ 0\leq \alpha\leq 1 \right\rbrace, \]
	which is the line segment between $(0,0,1)$ and $\left( \frac 12,\frac 12,0\right)$. Now consider the following diagonal matrices, which clearly form an orthonormal basis of real diagonal matrices for $\calD_3(\C)$:
\[\left. \begin{array}{lll}
	B_1=\text{Diag}\left( (1,0,0)\right)\\		
	B_2=\text{Diag}\left( \left(0,\frac{1}{\sqrt{2}},-\frac{1}{\sqrt{2}}\right)\right)\\
	B_3=\text{Diag}\left( \left(0,\frac{1}{\sqrt{2}},\frac{1}{\sqrt{2}}\right)\right)\\	
\end{array}\right\rbrace \Rightarrow C_{\mathbb{E},\mathbb{B}}=\begin{pmatrix}
1&0&0\\
0&\frac{1}{\sqrt{2}}&-\frac{1}{\sqrt{2}}\\
0&\frac{1}{\sqrt{2}}&\frac{1}{\sqrt{2}}\\
\end{pmatrix}.
\]
Observe that $C_{\mathbb{E},\mathbb{B}}$ is a rotation matrix of 45° about the $x$-axis. In this case, we obtain that
\begin{align*}
	m_S^{\mathbb{B}}&=\left\lbrace \begin{pmatrix}
		1&0&0\\
		0&\frac{1}{\sqrt{2}}&-\frac{1}{\sqrt{2}}\\
		0&\frac{1}{\sqrt{2}}&\frac{1}{\sqrt{2}}\\
	\end{pmatrix}\begin{pmatrix}
		\frac{\alpha}{2}\\
		\frac{\alpha}{2}\\
		1-\alpha
	\end{pmatrix}:\  0\leq \alpha\leq 1\right\rbrace\\
	&=\left\lbrace \left(\frac{\alpha}{2}, \frac{\sqrt{2}}{2}\left(\frac{3}{2}\alpha-1\right) , \frac{\sqrt{2}}{2}\left(1-\frac{\alpha}{2}\right)  \right) :\  0\leq \alpha\leq 1\right\rbrace.
\end{align*}
Therefore, $m_S^{\mathbb{B}}$ is a line segment rotation of $m_S^{\mathbb{E}}$ (see Figure \ref{graf moment}).

Now consider another subspace given by $V=\text{span}\{\frac{1}{\sqrt{2}}(1,1,0)\}$ and note that $S\perp V$. For this subspace, the moment is $m_V^{\mathbb{E}}=\{(\frac 12,\frac12,0)\}$ and $m_V^{\mathbb{E}}\cap m_S^{\mathbb{E}} = \{(\frac 12,\frac12,0)\}\neq \emptyset$. Then, by Theorem \ref{teo equiv JNR y minimal}, $M_{\lambda}= \lambda(P+Q)-\lambda (I-P-Q)$ is $\calD_3(\C)$-minimal and invertible with $\|M_\lambda\|=\lambda>0$, that is
\[M_{\lambda}=\lambda(2P+2Q-I)=\lambda\begin{pmatrix}
	0&1&0\\
	1&0&0\\
	0&0&1
\end{pmatrix}.\]
\begin{center}
	\begin{figure}[h!] 
	\includegraphics[width=.7\textwidth]{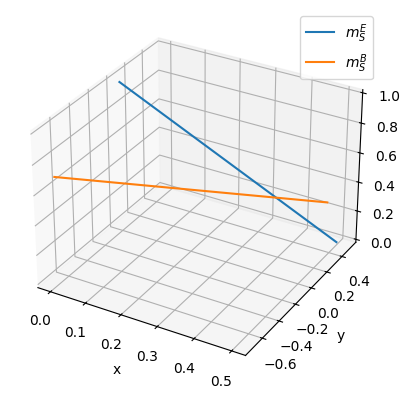}
	\caption{Moments in different basis of real diagonals of $3\times 3$}
	\label{graf moment}
	\end{figure}
\end{center}
\end{example}

\begin{example} [A basis of diagonal real matrices using Pauli strings] \label{quantum ex}
	Consider $\C^4$ and the following diagonal matrices 
	\[\begin{array}{lll}
		B_1=\frac 12I_4&B_2=\frac 12\text{Diag}\left( (1,1,-1,-1)\right)\\		
		B_3=\frac 12\text{Diag}\left( (1,-1,1,-1)\right)&B_4=\frac 12\text{Diag}\left( (1,-1,-1,1)\right).\\
	\end{array}
	 \]
	The set $\mathbb{B}=\{B_1;B_2;B_3;B_4\}$ is an orthonormal basis of $\mathcal{B}_4$ since 
	\[\tr(B_iB_j)=\left\lbrace \begin{array}{lll}
		0&\text{if}& i\neq j\\
		1&\text{if}& i=j
	\end{array} \right. \]
	for every $1\leq i,j\leq 4$. 
		If we consider the isomorphism $\C^4\cong \C^2\otimes \C^2$, where $\otimes$ denotes the Kronecker product, we obtain that each element of the basis $\mathbb{B}$ can be identified with a Kronecker product matrix as follows:
	\begin{align*}
		2B_1&= I_2\otimes I_2= I_4\\
		2B_2&=Z\otimes I_2= \text{Diag}\left( (1,1,-1,-1)\right) \\
		2B_3&=I_2\otimes Z= \text{Diag}\left( (1,-1,1,-1)\right)\\
		2B_4&=Z\otimes Z= \text{Diag}\left( (1,-1,-1,1)\right),
	\end{align*}
	where $I_2=\begin{pmatrix}
		1&0\\
		0&1
	\end{pmatrix}$ and $Z=\begin{pmatrix}
		1&0\\
		0&-1
	\end{pmatrix}$ are the identity and the Pauli $2\times 2$ diagonal matrix, respectively. The set $\mathbb{O}=\{I_4;2B_2;2B_3;2B_4\}$ is the Quantum computational basis for diagonal Hermitian Hamiltonians acting on $2$ qubits (Pauli strings). So, given a non diagonal Hamiltonian $H_0$, there exists a Hamiltonian $B_0$ which is a best diagonal approximation of $H_0$ in the spectral norm.

	In this basis, the moment for any subspace $S\subseteq \C^4$ is
	\[m_S^{\mathbb{B}}=C_{\mathbb{E},\mathbb{B}}m_S^{\mathbb{E}}\]
	with $C_{\mathbb{E},\mathbb{B}}=C_{\mathbb{B},\mathbb{E}}^t= \frac 12\begin{pmatrix}
		1&1&1&1\\
		1&1&-1&-1\\
		1&-1&1&-1\\
		1&-1&-1&1
	\end{pmatrix}^t= C_{\mathbb{B},\mathbb{E}}$.
	
	Take $S=\text{span}\{e_1;e_2\}$, $\calD_{S}=\{Y\in \calD:\ Y=\alpha e_1^*e_1+(1-\alpha)e_2^*e_2,\ 0\leq \alpha\leq 1\}$ and
	\[m_S^{\mathbb{E}}=\Diag(\calD_{S})=\{(\alpha,1-\alpha,0,0):\ 0\leq \alpha\leq 1\}.\]
	Then, 
	\[m_S^{{\mathbb{B}}}= \frac 12\begin{pmatrix}
		1&1&1&1\\
		1&1&-1&-1\\
		1&-1&1&-1\\
		1&-1&-1&1
	\end{pmatrix}\begin{pmatrix}
	\alpha\\
	1-\alpha\\
	0\\
	0
	\end{pmatrix}=\frac{1}{2}\begin{pmatrix}
	1\\
	1\\
	2\alpha-1\\
	2\alpha-1
	\end{pmatrix}\]
		
\end{example}

\begin{example}[a $C^*$-subalgebra of block operators]
	Consider $\calB= \calD_2(\C)\oplus M_2(\C)\subset M_4(\C)$. Any $B\in \calB$ can be written as
	\[
	B = \left[ 
	\begin{array}{c|c}
		D & 0
		\\
		\hline
		0 & T
	\end{array} 
	\right]= \left[ 
	\begin{array}{c|c}
		\begin{matrix}
			b_{11}&0\\
			0&b_{22}
		\end{matrix} & \begin{matrix}
			0&0\\
			0&0
		\end{matrix} \\
		\hline
		\begin{matrix}
			0&0\\
			0&0
		\end{matrix} & \begin{matrix}
		b_{33}&b_{34}\\
		\overline{b_{34}}&b_{44}
		\end{matrix}
	\end{array} 
	\right],\ b_{ii}\in \R, \ b_{34}\in \C.
	\]
	An orthonormal basis for $\calB$ is $\mathbb{B}=\{E_1,E_2,E_3,E_4, W^{3,4}, W^{4,,3}\}$, where
	\[
	W^{3,4} = \left[ 
	\begin{array}{c|c}
		\begin{matrix}
			0&0\\
			0&0
		\end{matrix} & \begin{matrix}
			0&0\\
			0&0
		\end{matrix} \\
		\hline
		\begin{matrix}
			0&0\\
			0&0
		\end{matrix} & \begin{matrix}
			0&\frac{1}{\sqrt{2}}\\
			\frac{1}{\sqrt{2}}&0
		\end{matrix}
	\end{array} 
	\right]\text{ and }\quad W^{4,3} = \left[ 
	\begin{array}{c|c}
		\begin{matrix}
			0&0\\
			0&0
		\end{matrix} & \begin{matrix}
			0&0\\
			0&0
		\end{matrix} \\
		\hline
		\begin{matrix}
			0&0\\
			0&0
		\end{matrix} & \begin{matrix}
			0&-\frac{i}{\sqrt{2}}\\
			\frac{i}{\sqrt{2}}&0
		\end{matrix}
	\end{array} 
	\right].
	\]
	Accordingly to Proposition \ref{prop 1} the moment of a subspace $S$ respect to this basis is
	\begin{align*}
		m_S^{\mathbb{B}}&=\{(\tr(YE_1),\tr(YE_2),\tr(YE_3),\tr(YE_4), \tr(YW^{3,4}),\tr(YW^{4,3})): Y\in \calD_{S}\}\\
		&=\{(Y_{11}, Y_{22},Y_{33},Y_{44}, \frac{1}{\sqrt{2}}(Y_{34}+Y_{43}), \frac{i}{\sqrt{2}}(Y_{34}-Y_{43}): Y\in \calD_{S}\}\subset \R^6.
	\end{align*}
	If $S=\text{span}\{v\}$, $v=(v_1,v_2,v_3,v_4)\in \C^4$ unitary
	\[m_S^{\mathbb{B}}=\left\{\left(|v_1|^2,|v_2|^2,|v_3|^2,|v_4|^2,\frac{2}{\sqrt{2}}\text{Re}(\overline{v_3}v_4),-\frac{2}{\sqrt{2}}\text{Im}(\overline{v_3}v_4) \right)\right\}.\]
	In particular, for $v=\frac 12(1,1,1,1)$
	\[m_S^{\mathbb{B}}=\left\{\left(\frac 14,\frac 14,\frac 14,\frac 14,\frac{1}{2\sqrt{2}},-\frac{1}{2\sqrt{2}}\right)\right\}.\]
	Let $V=\text{span}\{\frac 12(-1,-1,1,1)\}$, then $V\perp S$ and $m_S^{\mathbb{B}}=m_V^{\mathbb{B}}$. Using the orthogonal rank one projectors
		\begin{align*}
		P&=\frac{1}{2}(1,1,1,1)\otimes \frac{1}{2}(1,1,1,1)=\frac 14\begin{pmatrix}
			1&1&1&1\\
			1&1&1&1\\
			1&1&1&1\\
			1&1&1&1\\
			\end{pmatrix},\\		
		Q&=\frac{1}{2}(-1,-1,1,1)\otimes \frac{1}{2}(-1,-1,1,1)=\frac 14\begin{pmatrix}
			1&1&-1&-1\\
			1&1&-1&-1\\
			-1&-1&1&1\\
			-1&-1&1&1\\
		\end{pmatrix}.
	\end{align*}
	By Theorem \ref{teo equiv JNR y minimal}, any matrix of the form $M_{\lambda,\mu}= \lambda(P-Q)+\mu(I-P-Q)$, with $0<\mu<\lambda$ is a $\calB$-minimal matrix
	\[M_{\lambda,\mu}=\lambda\begin{pmatrix}
		0&0&\frac 12&\frac 12\\
		0&0&\frac 12&\frac 12\\
		\frac 12&\frac 12&0&0\\
		\frac 12&\frac 12&0&0\\
	\end{pmatrix}+\mu \begin{pmatrix}
	\frac 12&\frac 120&0&\\
	\frac 12&\frac 120&0&\\
	0&0&\frac 12&\frac 12\\
	0&0&\frac 12&\frac 12\\
	\end{pmatrix}.\]
\end{example}

\end{document}